\documentclass[reqno,12pt,a4paper]{amsart}

\usepackage{enumerate}
\usepackage{amstext,amsmath,amsthm,amsfonts,amssymb,amscd}
\usepackage{latexsym,mathrsfs,dsfont,euscript}
\usepackage{fullpage}
\usepackage{txfonts}
\usepackage{color}

\usepackage{hyperref} 
\hypersetup{
  colorlinks,%
  citecolor=blue,%
    filecolor=red,%
    linkcolor=red,%
    urlcolor=blue
}

\newtheorem{theorem}{Theorem}

\newtheorem{lemma}[theorem]{Lemma}
\newtheorem{corollary}[theorem]{Corollary}
\theoremstyle{definition}

\theoremstyle{remark}

\numberwithin{equation}{section}

\DeclareMathOperator{\supp}{supp}
\newcommand{\abs}[1]{\left\vert#1\right\vert}
\newcommand{\set}[1]{\left\{#1\right\}}
\newcommand{\proin}[2]{\left<#1,#2\right>}
\newcommand{\norm}[1]{\left\Vert#1\right\Vert}

\allowdisplaybreaks
\begin{document}
\title[]{Function spaces of coercivity for the fractional Laplacian in spaces of homogeneous type}
%


\author[]{Hugo Aimar}
\author[]{Ivana G\'{o}mez}
\thanks{This work was supported by the CONICET; ANPCyT-MINCyT (grants PICT-2568,2012; PICT-3631,2015); and UNL (grant CAID-50120110100371LI,2013)}
%


\begin{abstract}
We combine dyadic analysis through Haar type wavelets defined on Christ's families of generalized cubes, and Lax-Milgram theorem, in order to prove existence of Green's functions for fractional Laplacians on some function spaces of vanishing small resolution in spaces of homogeneous type.
\end{abstract}
\maketitle

\section{Introduction}
The most elementary view of fractional Laplacian operators in the Euclidean space $\mathbb{R}^n$ is provided by the Fourier transform. The symbol of a linear PDE with constant coefficients $\sum_{\abs{\alpha}\leq m}c_\alpha \partial^\alpha$, is the polynomial $P(\xi)=\sum_{\abs{\alpha}\leq m}c_\alpha (2\pi i\xi)^\alpha$ if we are considering the function $\widehat{\varphi}(\xi)=\int_{\mathbb{R}^n}e^{2\pi i\xi\cdot x}\varphi(x) dx$ as the Fourier transform of $\varphi$. In some cases, such as the Laplacian, $P$ is a radial function. For $-\tfrac{1}{4\pi^2}\Delta$, $P(\xi)=\abs{\xi}^2$. For $0<s<1$, the powers $P^s(\xi)=\abs{\xi}^{2s}$ of $P(\xi)$ are also symbols for operators which are well defined on smooth functions. Actually the most explicit kernel realization of such operator is given by
\begin{equation*}
(-\triangle)^{s}\varphi(x)=p.v.\int_{\mathbb{R}^n}\frac{\varphi(x)-\varphi(y)}{\abs{x-y}^{n+2s}}dy
\end{equation*}
with $\varphi$ smooth. The principal value is taken in the usual way of symmetric truncation about the origin. For $0<s<\tfrac{1}{2}$ the integral is absolutely convergent and there is no need for the principal value. When for $s<1$ these operators loose the local character of the case $s=1$. Several important facts for the Laplacian remain in the family $(-\triangle)^s$ ($0<s<1$). One of them is its role in the variational approach of the theory. In fact, $(-\triangle)^s$ is the Euler-Lagrange operator associated to the energy given by the Dirichlet bilinear form
\begin{equation*}
\int_{\mathbb{R}^n}\int_{\mathbb{R}^n}\frac{(u(x)-u(y))(v(x)-v(y))}{\abs{x-y}^{n+2s}}dx dy.
\end{equation*}
There are several points of view for $(-\triangle)^s$. The probabilistic approach is provided by the substitution of the Wiener process by the general L\'{e}vy process to generate the diffusion based on $(-\triangle)^s$. The analytic point of view, which in the case $s=\tfrac{1}{2}$ leads to the Dirichlet to Neumann operator, received an extraordinary impulse with the work of Caffarelli and Silvestre \cite{CaSi07}. In \cite{CaChaVa11} and \cite{CaSir18} several aspects of the theory are considered for nonlocal operators with kernels that are controlled by $\abs{x-y}^{-n-2s}$. Also the semigroup approach in \cite{StiTo10} provides an important tool of analysis. In this paper we explore some aspects of this robustness of the theory in a different direction; the underlying geometric setting. In metric measure spaces and in particular in spaces of homogeneous type. When the space has a well defined dimension, as is the case of Ahlfors $\gamma$-regular spaces, the Dirichlet form, the associated energy and the derived Euler-Lagrange operator are well defined at least for small $s>0$.
The interest for such a general setting is perhaps provided by some non differentiable structures like self similar fractals.

Let us briefly describe a non Euclidean paradigmatic situation of our context. The Sierpinski quadrant. Let $S$ be the Sierpinski triangle with vertices at the points $\vec{0}=(0,0)$, $\vec{e_1}=(1,0)$ and $\vec{e_2}=(0,1)$ of $\mathbb{R}^2$. The set $S$ is the only fixed point of the iterated function system induced by the affine contractions defined at the point $(x,y)\in \mathbb{R}^2$ by $F_1(x,y)=\tfrac{1}{2}(x,y)$, $F_2(x,y)=(\tfrac{1}{2}+\tfrac{1}{2}x,\tfrac{1}{2}y)$ and $F_3(x,y)=(\tfrac{1}{2}x,\tfrac{1}{2}+\tfrac{1}{2}y)$. Let $X=\cup_{m=1}^{\infty}2^m S$. The Haussdorf dimension of $X$ is $\gamma=\tfrac{\log 3}{\log 2}>1$. Moreover, if $\mu$ denotes the restriction to the Borel subsets of $X$ of the Hausdorff measure of dimension $\gamma$, we have that $\mu(B(x,r))\simeq r^\gamma$ for every $x\in X$, every $r>0$, with $B(x,r)=\{y\in X:\abs{x-y}<r\}$. In the terminology that we shall soon review, $(X,d,\mu)$ is a Ahlfors $\gamma$-regular space of homogeneous type with $d$ the restriction to $X\times X$ of the Euclidean distance.

There exist in $X$ also some natural \textit{dyadic sets} and corresponding Haar wavelets as orthonormal basis for $L^2(X,\mu)$. In the proof of our main result the role of Haar dyadic analysis will be crucial. So, let us describe briefly the dyadic sets and the Haar system in the Sierpinski quadrant $(S,d,\mu)$.

The basic Sierpinski triangle $S$, contained in the plane triangle $T$ with vertices at $\vec{0}$, $\vec{e_1}$ and $\vec{e_2}$, can be obtained as the limit, in the sense of Haussdorf, of the iteration of the operation $\Phi(A)=\cup_{i=1}^{3}F_i(A)$ starting with $A=T$. With $\Phi^m$ we denote the $m$-iteration of $\Phi$. For each $j\in Z$ the set $A_j=\cup_{m-k=j}2^k\Phi^m(T)$ is a union of non overlapping triangles of size $2^{-j}$. Set $T^j_l$ to denote these triangles with $l\in \mathbb{N}$ and $Q^j_l=T^j_l\cap X$. These sets in $X$ satisfies all the desired properties of nested partition of the space $X$ with a metric control and homogeneity in the number of offspring of each cube. Each $Q^j_l$ divides in exactly three \textit{cubes} of the next scale refinement level $Q^{j+1}_{l,1}$, $Q^{j+1}_{l,2}$ and $Q^{j+1}_{l,3}$. Since the space $V^j_l$ of real functions defined on $Q^j_l$ that are constant on each $Q^{j+1}_{l,i}$, $i=1,2,3$, has dimension three, a basis for $V^j_l$ is given $\{1,h^j_{l,1}, h^j_{l,2}\}$ where $1$ is the constant function on $Q^j_l$,
\begin{equation*}
h^j_{l,1}=\frac{4}{\sqrt{42}}3^{\tfrac{j}{2}}\left(\mathcal{X}_{Q^{j+1}_{l,1}}+\frac{1}{4}\mathcal{X}_{Q^{j+1}_{l,2}}-\frac{5}{4}\mathcal{X}_{Q^{j+1}_{l,3}} \right)
\end{equation*}
and 
\begin{equation*}
h^j_{l,2}=\frac{3}{\sqrt{14}}3^{\tfrac{j}{2}}\left(-\frac{2}{3}\mathcal{X}_{Q^{j+1}_{l,1}}+\mathcal{X}_{Q^{j+1}_{l,2}}-\frac{1}{3}\mathcal{X}_{Q^{j+1}_{l,3}} \right)
\end{equation*}
The countable family $\mathscr{H}=\{h^j_{l,1}, h^j_{l,2}:j\in \mathbb{Z}, l\in\mathbb{N}\}$ is an orthonormal basis for $L^2(X,\mu)$.

Aside from its generality as a model for many problems in the sciences, variational methods show also a remarkable geometric robustness. In the above setting $(X,d,\mu)$, a Dirichlet form associated to an energy norm is naturally defined,
\begin{equation*}
\int_X\int_X\frac{u(x)-u(y)}{d^s(x,y)}\frac{v(x)-v(y)}{d^s(x,y)}\frac{d\mu(x)d\mu(y)}{d^\gamma (x,y)}
\end{equation*}
also the associated energy
\begin{equation}\label{eq:energyfirst}
\iint_X\left(\frac{u(x)-u(y)}{d^s(x,y)}\right)^2\frac{d\mu(x)d\mu(y)}{d^\gamma (x,y)},
\end{equation}
are well defined at least for $u$ and $v$ compactly supported and with H\"{o}lder regularity exponent larger than $s$ (see Lemma~\ref{lemma:welldefinedenergy} below). 

The Euler-Lagrange equation associated to the minimization of the energy \eqref{eq:energyfirst} on an adequate subspace of the $L^2(X,\mu)$ functions with finite energy, is given by
\begin{equation*}
D^{2s}_du(x)=\int_X\frac{u(x)-u(y)}{d^{2s}(x,y)}\frac{d\mu(y)}{d^{\gamma}(x,y)}.
\end{equation*}
The search of a Green function for $D^{2s}$, submitted to some specific boundary type condition, through Lax-Milgram methods requires the analysis of the  coercivity of the bilinear form. In this note we provide subspaces of coercivity for Dirichlet form using Haar wavelet methods built on Christ'dyadic families in $(X,d,\mu)$. 

The paper is organized in brief sections that collect the main results of each topic which we shall use to accomplish our aims. Section~\ref{sec:GeometricSetting} contains the preliminary definitions. In Section~\ref{sec:energyandspacescoercivity} we introduce the generalized energy on any space of homogeneous type and we prove the basic result relating this energy with the dyadic one. Then we use the characterization of the dyadic energy in terms of the Haar systems proved in \cite{AcAiBoGo16} and we propose the spaces of coercivity. Section~\ref{sec:finiteenergyregularity} is devoted to a basic analysis of H\"{o}lder regularity of functions with finite energy when the space is Ahlfors $\gamma$-regular as a consequence of the results in \cite{MaSe79Lip}. In Section~\ref{sec:weaksolutionsandGreenfunction} we apply the results of sections~\ref{sec:energyandspacescoercivity} and~\ref{sec:finiteenergyregularity} and Lax-Milgram Theorem in order to prove the existence of Green's type functions on those spaces of coercivity.

\section{Geometric setting. Dyadic and Haar systems}\label{sec:GeometricSetting}
Let us briefly introduce in this section the basic definitions of space of homogeneous type, Ahlfors $\gamma$-regular space, Christ's dyadic families, Haar wavelets and dyadic distance.

A space of homogeneous type is a set $X$ endowed with two structures. A quasi-distance $d$ on $X$ is a symmetric, nonnegative real function on $X\times X$ such that $d(x,y)=0$ if and only if $x=y$ and for some constant $\kappa\geq 1$, the inequality $d(x,y)\leq\kappa (d(x,z)+d(z,y))$ holds for every $x$, $y$ and $z$ in $X$. The second structure is given by a positive measure $\mu$ an a $\sigma$-algebra containing the $d$-balls. Both structure are related by the doubling property; there exists a constant $A>0$ such that $0<\mu(B(x,2r))\leq A\mu(B(x,r))<\infty$ for every $x\in X$ and every $r>0$. Here $B(x,r)$ denotes the open ball $\{y\in X: d(x,y)<r\}$. Along this paper we shall assume that the space has no atoms ($\mu(\{x\})=0$ for every $x\in X$ and that $\mu(X)=+\infty$). The most important structure properties of spaces of homogeneous type and quasi-metrics can be found in \cite{MaSe79Lip}. On the other hand, the construction of metric controlled dyadic families in spaces of homogeneous type, due to M. Christ, can be found in \cite{Christ90}. Let us sketch the properties of these families here,
\begin{itemize}
	\item [(D1)] $\mathcal{D}=\cup_{j\in \mathbb{Z}}\mathcal{D}_j$;
	\item [(D2)] there exists $0<\nu<1$ such that for every cube $Q\in\mathcal{D}_j$ we have that diameter\, $Q\simeq \nu^j$ and eccentricity\, $Q\simeq 1$. Here, as usual $diam\, Q=\sup\{d(x,y): x,y\in Q\}$ and $eccentricity\, Q=\sup\{\tfrac{r(B_1)}{r(B_2)}: B_1\subseteq Q\, and\, B_2\supseteq Q\}$ where $B_i$, $i=1,2$ are $d$-balls;
	\item [(d3)] each $Q\in\mathcal{D}$ is an open set;
	\item [(d4)] except for a set of $\mu$-measure zero each $\mathcal{D}_j$ is a partition of $X$;
	\item [(D5)] for each $Q\in\mathcal{D}_{j+1}$ there exists one and only one $\widetilde{Q}\in\mathcal{D}_j$ such that $Q\subseteq\widetilde{Q}$;
	\item [(D6)]for some geometric constant $M$ and every $Q$ in $\mathcal{D}_j$ we have that $1\leq\#\vartheta(Q)\leq M$ where $\vartheta(Q)=\{Q'\in\mathcal{D}_{j+1}:Q'\subseteq Q\}$ is the offspring of $Q$;
	\item [(D7)] given $Q$ and $\bar{Q}$ in $\mathcal{D}$ then $\bar{Q}\cap Q=\emptyset$, $\bar{Q}\subseteq Q$ or $Q\subseteq\bar{Q}$;
	\item [(D8)] $X$ is a quadrant for $\mathcal{D}$, in other words, for every $Q\in\mathcal{D}$ the union of all those $\hat{Q}$ in $\mathcal{D}$ containing $Q$ coincides with $X$.
\end{itemize}
After the application a the procedure of closure and difference properties (d3) and (d4) can be changed into
\begin{itemize}
	\item [(D3)] each cube $Q$ in $\mathcal{D}$ is a Borel set;
	\item [(D4)] each $\mathcal{D}_j$ is a partition of $X$.	
\end{itemize}
See also \cite{AiBeIa07}.

Once such a dyadic family $\mathcal{D}$ is given on $X$ we also have Haar orthonormal bases for $L^2(X,\mu)$. Actually these bases are built on a multiresolution analysis (MRA) induced by the sequence $\mathcal{D}_j$ of dyadic cubes of level $j$. For a given integer $j\in\mathbb{Z}$ define
\begin{equation*}
V_j=\{f\in L^2(X,\mu): f\textrm{ restricted to } Q \textrm{ is constant for every } Q\in \mathcal{D}_j\}.
\end{equation*}
The properties of the family $\mathcal{D}$ induce the following properties for the MRA sequence $\{V_j: j\in\mathbb{Z}\}$.
\begin{itemize}
	\item [(MRA1)] $V_j\subset V_{j+1}$ for every $j\in\mathbb{Z}$;
	\item [(MRA2)] $\overline{\cup_{j\in\mathbb{Z}}V_j}=L^2(X,\mu)$;
	\item [(MRA3)] $\cap_{j\in\mathbb{Z}}V_j=\{0\}$;
	\item [(MRA4)] $\set{\frac{\mathcal{X}_Q}{\sqrt{\mu(Q)}}: Q\in\mathcal{D}_j}$ is an orthonormal basis for $V_j$ for every $j\in\mathbb{Z}$. As usual we shall denote by $P_j$ the orthonormal projector of $L^2(X,\mu)$ onto $V_j$.
\end{itemize}
For a given $Q\in \mathcal{D}$, set $\vartheta(Q)$ to denote the offspring of $Q$. That is $\vartheta(Q)$ is the family of all $Q'\in\mathcal{D}_{j(Q)+1}$ with $Q'\subseteq Q$ and $j(Q)$ the level of $Q$. From (D6) we have that $1\leq\#\vartheta(Q)\leq M$ for some geometric constant $M$. Let $Q$ be such that $\#(\vartheta(Q))>1$. Set $\mathcal{V}_Q$ to denote the finite dimensional (at most $M$) subspace of those functions in $L^2(X,\mu)$ vanishing outside $Q$ and which are constant on each $Q'\in\vartheta(Q)$. An algebraic basis for $\mathcal{V}_Q$ is given by $\left\{\tfrac{\mathcal{X}_Q}{\sqrt{\mu(Q)}}\right\}\cup\left\{\mathcal{X}_{Q'}:Q'\in\widetilde{\vartheta(Q)}\right\}$ with $\widetilde{\vartheta(Q)}$ any subset of $\vartheta(Q)$ with $\#\vartheta(Q)-1$ elements. If we orthonormalize this basis for $\mathcal{V}_Q$ preserving the first function $\tfrac{\mathcal{X}_Q}{\sqrt{\mu(Q)}}$, we get the basis $\mathcal{B}(Q)=\left\{\tfrac{\mathcal{X}_Q}{\sqrt{\mu(Q)}}\right\}\cup\left\{h_l: l=1,\ldots,\#\vartheta(Q)-1\right\}$. Set $\mathscr{H}=\bigcup_{Q\in\mathcal{D}}\mathscr{H}(Q)$, where $\mathscr{H}(Q)=\{h_l: l=1,\ldots,\#\vartheta(Q)-1\}$. Then $\mathscr{H}$ is an orthonormal basis for $L^2(X,\mu)$. Notice that each $h$ has support in $Q$ and mean value zero. Hence
\begin{equation*}
\norm{f}^2_{L^2(X,\mu)}=\sum_{h\in\mathscr{H}}\abs{\proin{f}{h}}^2
\end{equation*}
and $\mathscr{H}$ is an unconditional basis for $L^p(X,\mu)$ for every $1<p<\infty$. We shall write $Q(h)$ to denote the cube in which $h$ is based in the above sense. Also $j(h)$ denotes the level of $Q(h)$.

For a $\gamma>0$, an Ahlfors $\gamma$-regular space $(X,d,\mu)$ is a metric space $(X,d)$ with a Borel measure $\mu$ such that $c_1r^{\gamma}\leq\mu(B(x,r))\leq c_2r^{\gamma}$ for some positive constants $c_1$ and $c_2$, every $x\in X$, and every $r>0$. Ahlfors condition implies that the space is of homogeneous type and unbounded and that $\mu(\{x\})=0$ for every point $x\in X$.

An elementary but central fact for the finiteness of the Dirichlet form, the corresponding energy and the associated Euler-Lagrange operator is provided by the local and global integrability of powers of $d$. 
\begin{lemma}\label{lem:localglobalintegrability}
	Let $(X,d,\mu)$ be Ahlfors $\gamma$-regular, then
\begin{enumerate}[(a)]
	\item
	$\int_{B(x,r)}\frac{d\mu(y)}{d^{\gamma+s}(x,y)}=	\int_{d(x,y)\geq r}\frac{d\mu(y)}{d^{\gamma-s}(x,y)}=+\infty$, for every $x\in X$, $r>0$ and $s\geq 0$; and
\item $\int_{B(x,r)}\frac{d\mu(y)}{d^{\gamma-s}(x,y)}\simeq r^s$ and	$\int_{d(x,y)\geq r}\frac{d\mu(y)}{d^{\gamma+s}(x,y)}\simeq r^{-s}$, for every $x\in X$, $r>0$ and $s>0$.
\end{enumerate}
\end{lemma}
We use the notation $A\simeq B$ if the quotient $\tfrac{A}{B}$ is bounded above and bellow by geometric constants.
The proof follow by dyadic decomposition of each integral and the $\gamma$- regularity of the space.

In the next statement we introduce the dyadic distance induced by a dyadic family $\mathcal{D}$ on $X$.
\begin{lemma}\label{lemma:dyadicdistance}
	Let $(X,d,\mu)$ be a space of homogeneous type with no atoms ($\mu(\{x\})=0, x\in X$).
	The function $\delta:X\times X\to\mathbb{R}^+$ defined as zero on the diagonal and\begin{equation*}
	\delta(x,y)=\inf\{\mu(Q): Q\in\mathcal{D}, x\in Q\, and\, y\in Q\}
	\end{equation*}
	is a distance on $X$ such that $(X,\delta,\mu)$ is a Ahlfors $1$-regular space. Moreover, for some constant $C>0$ we have that 
	$\mu(B(x,d(x,y)))\leq C\delta(x,y)$. When $(X,d,\mu)$ is Ahlfors $\gamma$-regular, we also have $d^\gamma(x,y)\leq C\delta(x,y)$.
\end{lemma}	
\begin{proof}
	Let us first check the triangle inequality. Let $x$, $y$ and $z$ be three given points in $X$. Let $Q(x,y)$ be the smallest cube in $\mathcal{D}$ such that $x$ and $y$ belong to $Q(x,y)$. Similarly, set $Q(y,z)$ to denote the smallest cube in $\mathcal{D}$ containing $y$ and $z$. Since $y\in Q(x,y)\cap Q(y,z)$, by (D7) we must have $Q(x,y)\subset Q(y,z)$ or $Q(y,z)\subset Q(x,y)$. Hence $x$ and $z$ belong both to $Q(y,z)$ or to $Q(x,y)$. This fact implies that $\delta(x,y)\leq \sup\{\mu(Q(y,z)),\mu(Q(x,y))\}=\sup\{\delta(y,z),\delta(x,y)\}$, and $\delta$ is in fact an ultrametric.
	
	Notice that if $x\neq y$ are two points in $X$, since $d(x,y)>0$, we have, from property (D2), that for some large $j$, $x$ and $y$ can not belong to the same $Q\in\mathcal{D}_j$. Hence $\delta(x,y)>0$.
	
	Let us describe the $\delta$- ball centered at $x\in X$ with radious $r>0$. Notice that $B_\delta(x,r)=\{y:\delta(x,y)<r\}=\cup_{\{Q\in\mathcal{D}: x\in Q, \mu(Q)<r\}} Q$ which is actually a the largest dyadic cube $Q^r_x$ containing $x$ with measure less than $r$. Hence $\mu(B_\delta(x,r))=\mu(Q^r_x)<r$. On the other hand, since $Q^r_x$ is the largest cube containing $x$ with measure less than $r$, $\widetilde{Q^r_x}$ the first ancestor of $Q^r_x$, must to satisfy $\mu(\widetilde{Q^r_x})\geq r$. From (D2) and the doubling property for $d$-balls we obtain $\mu(Q^r_x)\geq c\mu(\widetilde{Q^r_x})\geq cr$ for some $0<c<1$. In other words, $cr\mu(B_\delta(x,r))\leq r$ for every $x\in X$ and every $r>0$.
	
	Let us prove that $\mu(B(x,d(x,y)))\leq C\delta(x,y)$, for some constant $C$. Both sides of the inequality vanish if $x=y$ because $X$ has no atoms. Assume $x\neq y$. Let $Q$ be the smallest dyadic cube in $\mathcal{D}$ such that $x$ and $y$ belong to $Q$. Then, with the notation in (D2), $B_1=B(x_1,r_1)$ and $B_2=B(x_1,mr_1)$, $\delta (x,y)=\mu(Q)\geq\mu(B_1)$. On the other hand, $B(x,d(x,y))\subset m\kappa(2\kappa+1)B_1$. In fact, with $B_1=B(x_1,r_1)$, $B_2=B(x_1,mr_1)$, $m$ constant, we have that, for $z\in B(x,d(x,y))$,
	\begin{align*}
	d(z,x_1)&\leq\kappa[d(z,x)+d(x,x_1)]\\
	&<\kappa[d(x,y)+mr_1]\\
	&\leq\kappa[\kappa(d(x,x_1)+d(x_1,y))+mr_1]\\
	&<m\kappa[2\kappa+1]r_1.
	\end{align*}
The last statement follows from the above and the fact that $\mu(B(x,d(x,y)))\simeq d^\gamma(x,y)$ in the case of Ahlfors $\gamma$-regularity.
\end{proof}

Let us point out that the inequality $\delta^\gamma\leq cd$ does not hold even in Euclidean settings. Since $(X,\delta,\mu)$ is Ahlfors $1$-regular space, then satisfies the integral properties of powers of $\delta$ contained in Lemma~\ref{lem:localglobalintegrability} with $\gamma=1$.

\section{Energy and spaces of coercivity}\label{sec:energyandspacescoercivity}
Let $(X,d,\mu)$ be a nonatomic space of homogeneous type with $\mu(X)=+\infty$. In the spirit of \cite{GoKoSha10} and references there in, for generally non necessarily Ahlfors regular spaces, it is natural to consider a generalized notion of energy given by
\begin{equation*}
\mathscr{E}^{d,\mu}_\sigma(u)=\iint_{X\times X}\frac{\abs{u(x)-u(y)}^2}{\mu(B(x,d(x,y)))^{1+2\sigma}}d\mu(x)d\mu(y)=\iint_{X\times X}\abs{\frac{u(x)-u(y)}{\mu(B(x,d(x,y)))^{\sigma}}}^2\frac{d\mu(x)d\mu(y)}{\mu(B(x,d(x,y)))},
\end{equation*}
for $\sigma>0$.

It is not clear, at first glance, whether or not there are in general nontrivial (non constant) functions with finite energy. Nevertheless, the subspace $H^\sigma_{d,\mu}$ of functions in $L^2(X,\mu)$ with finite energy is a Hilbert space with the norm
\begin{equation*}
\norm{u}_{H^\sigma_{d,\mu}}=\norm{u}_{L^2(\mu)}+\sqrt{\mathscr{E}^{d,\mu}_\sigma(u)}.
\end{equation*}
We say that a Hilbert subspace $H^{\sigma,0}_{d,\mu}$ of $H^{\sigma}_{d,\mu}$ is of coercivity for $\mathscr{E}^{d,\mu}_\sigma$ if there exists a constant $c$, depending on $H^{\sigma,0}_{d,\mu}$, such that the inequality
\begin{equation*}
\mathscr{E}^{d,\mu}_\sigma(u)\geq c\norm{u}^2_{L^2(\mu)}
\end{equation*}
holds for every $u\in H^{\sigma,0}_{d,\mu}$. In other words, $\sqrt{\mathscr{E}^{d,\mu}_\sigma(u)}$ on $H^{\sigma,0}_{d,\mu}$ becomes a norm, which is equivalent to $\norm{u}_{H^{\sigma}_{d,\mu}}$. In the search of spaces of type $H^{\sigma,0}_{d,\mu}$, we shall use dyadic analysis,  the estimate for $\mu(B(x,d(x,y)))$ in terms of $\delta(x,y)$ contained in Lemma~\ref{lemma:dyadicdistance} and a result in \cite{AcAiBoGo16} concerning the characterization of $H^{\sigma}_{\delta,\mu}$ in terms of Haar coefficients. Notice that from Lemma~\ref{lemma:dyadicdistance}, $H^{\sigma}_{\delta,\mu}$ is the subspace of those functions $u$ in $L^2(X,\mu)$ such that
\begin{equation*}
\mathscr{E}^{\delta,\mu}_\sigma(u)=\iint_{X\times X}\frac{\abs{u(x)-u(y)}^2}{\delta(x,y)^{1+2\sigma}}d\mu(x)d\mu(y)
\end{equation*}
is finite. Let us state precisely the characterization of $H^{\sigma}_{\delta,\mu}$ given in \cite{AcAiBoGo16}.
\begin{lemma}[Theorem~5 in \cite{AcAiBoGo16}]\label{lemma:AABG16}
Let $(X,d,\mu)$ be a non atomic space of homogeneous type with $\mu(X)=+\infty$. Let $\mathcal{D}$ be any dyadic system satisfying (D1) to (D8), let $\mathscr{H}$ be any Haar system associated to $\mathcal{D}$ and let $\delta$ be the dyadic distance induced by $\mathcal{D}$ on $X$. Then, for $0<\sigma<1$, the space $H^{\sigma}_{\delta,\mu}$ coincides with the subspace of $L^2(X,\mu)$ for which the series $\sum_{h\in\mathscr{H}}\frac{\abs{\proin{u}{h}}^2}{\mu(Q(h))^{2\sigma}}$ converges. Moreover, $\mathscr{E}^{\delta,\mu}_\sigma(u)=\sum_{h\in\mathscr{H}}\frac{\abs{\proin{u}{h}}^2}{\mu(Q(h))^{2\sigma}}$ and
	\begin{equation*}
	\norm{u}^2_{\sigma,\delta}\simeq \norm{u}^2_{L^2(X,\mu)}+\sum_{h\in\mathscr{H}}\frac{\abs{\proin{u}{h}}^2}{\mu(Q(h))^{2\sigma}}.
	\end{equation*}
\end{lemma}

For $\lambda>0$, let $M_\lambda$ be the Hilbert subspace of $L^2(X,\mu)$ generated by all the Haar functions $h\in\mathscr{H}$ with $\mu(Q(h))>\lambda$. In other words, $M_\lambda$ is the $L^2$ closure of the linear span of $\mathscr{H}=\{h\in\mathscr{H}:\mu(Q(h))>\lambda\}$. Let us write $\Pi_\lambda$ to denote the orthogonal projection of $L^2$ onto $M_\lambda$ and $\textrm{Ker }\Pi_\lambda$ to denote the kernel of $\Pi_\lambda$. The main result of this section is now an easy consequence of lemmas~\ref{lemma:dyadicdistance} and~\ref{lemma:AABG16}.
\begin{theorem}\label{thm:CoercivityPi}
	Let $(X,d,\mu)$ be a nonatomic space of homogeneous type with $\mu(X)=+\infty$. Then every space of the type $H^{\sigma}_{\delta,\mu}\cap \textrm{Ker } \Pi_\lambda$, $\lambda>0$, is a space of coercivity for $\mathscr{E}^{d,\mu}_\sigma$. In other words, for $u\in H^{\sigma}_{\delta,\mu}$ with $\Pi_\lambda u=0$, for some $\lambda$, we have the inequality
	\begin{equation*}
	\mathscr{E}^{d,\mu}_\sigma(u)\geq C\norm{u}^2_{L^2},
	\end{equation*}
	with $C$ depending on $\lambda$.
\end{theorem}
\begin{proof}
	Notice that if $\Pi_\lambda u=0$ and $\lambda'>\lambda$, then $\Pi_{\lambda'} u=0$. Hence, since $\mu(B(x,d(x,y)))\leq c\delta(x,y)$,
	\begin{align*}
\mathscr{E}^{d,\mu}_\sigma (u) 
&=\iint_{X\times X}\frac{\abs{u(x)-u(y)}^2}{\mu(B(x,d(x,y)))^{1+2\sigma}}d\mu(x)d\mu(y)\\
&\geq C \iint_{X\times X}\frac{\abs{u(x)-u(y)}^2}{\delta(x,y)^{1+2\sigma}}d\mu(x)d\mu(y)\\
&=C\mathscr{E}^{\delta,\mu}_\sigma (u)\\
&=C\sum_{h\in\mathscr{H}}\frac{\abs{\proin{u}{h}}^2}{\mu(Q(h))^{2\sigma}}\\
&=C\sum_{\{h\in\mathscr{H}:\mu(Q(h))\leq\lambda\}}\frac{\abs{\proin{u}{h}}^2}{\mu(Q(h))^{2\sigma}}\\ 	
&\geq\frac{C}{\lambda^{2\sigma}}\sum_{h\in\mathscr{H}}\abs{\proin{u}{h}}^2\\
&=\frac{C}{\lambda^{2\sigma}}\norm{u}^2_{L^2}.	
	\end{align*}
\end{proof}

It might be important to remark that in a general space of homogeneous type, not of regular Ahlfors type, the measure of cubes of the same level can be very different. This is the reason why we are using $\textrm{Ker }\Pi_\lambda$ instead of $\textrm{Ker } P_j$ in the above result. When the space is Ahlfors $\gamma$-regular the two approaches coincide. Elementary examples in which scales do not give a good control of the measure of the cubes, are given by Muckenhoupt weights. Take $X=\mathbb{R}^+$, $d$ the usual distance and $d\mu=wdx$ with $w(x)=x^{-1/2}dx$. The usual dyadic intervals in $\mathbb{R}^+$ satisfy (D1) to (D8) with respect to $d$. On the other hand, the measure of the intervals $I^j_k$ of a fixed level $j$ tends to zero as $k\to\infty$ ($I^j_k=[k2^{-j},(k+1)2^{-j})$).

\begin{corollary}\label{coro:CoercivityPj}
	Let $(X,d,\mu)$ be an Ahlfors $\gamma$-regular space. Then every space of the type $H^\sigma_{\delta,\mu}\cap \textrm{Ker } P_j$, $j\in \mathbb{Z}$, is a space of coercivity for $\mathscr{E}^{d,\mu}_\sigma$. In other words, for every $u\in H^\sigma_{d,\mu}$ with $P_j u=0$, for some $j\in\mathbb{Z}$, we have
	\begin{equation*}
	\mathscr{E}^{d,\mu}_\sigma (u)\geq C\norm{u}^2_{L^2},
	\end{equation*}
	with $C$ depending on $j$.
\end{corollary}
\begin{proof}
	From (D2) if $Q\in\mathcal{D}_j$ we have that $\mu(Q)\simeq\nu^{\gamma j}$, because of the Ahlfors $\gamma$-regularity of the space. Hence if $P_ju=0$ then for some $\lambda>0$ we also have that $\Pi_{\lambda}u=0$.
\end{proof}

Let us observe that, for a given bounded set $\Omega$ in $X$, $\{u\in H^\sigma_{d,\mu}: u$ vanishes outside $\Omega$ and $\int ud\mu=0\}$ is also a space of coercivity for $\mathscr{E}^{d,\mu}_\sigma$, since $\Omega\subset Q$ for some $Q\in\mathcal{D}$ and this space is a closed subspace of some $H^\sigma_{\delta,\mu}\cap \textrm{Ker }P_j$.

\section{Finite energy and regularity}\label{sec:finiteenergyregularity}
For a given Ahlfors $\gamma$-regular space with $\gamma>0$, we have that, with the notation of Section~\ref{sec:energyandspacescoercivity}, 
\begin{equation*}
\mathscr{E}^{d,\mu}_\sigma(u)=\iint_{X\times X}\frac{\abs{u(x)-u(y)}^2}{\mu(B(x,d(x,y)))^{1+2\sigma}}\mu(x)d\mu(y)\simeq
\iint_{X\times X}\frac{\abs{u(x)-u(y)}^2}{d(x,y)^{\gamma+2\gamma\sigma}}d\mu(x)d\mu(y).
\end{equation*}
Then 
\begin{equation*}
\mathscr{E}^{d,\mu}_{s/\gamma}(u)\simeq\iint_{X\times X}\left(\frac{\abs{u(x)-u(y)}}{d(x,y)^{s}}\right)^2 \frac{d\mu(x)d\mu(y)}{d(x,y)^\gamma}
\end{equation*}
and the space
\begin{equation*}
H^{s/\gamma}_{d,\mu}=\set{u\in L^2:\mathscr{E}^{d,\mu}_{s/\gamma}(u)<\infty}
\end{equation*}
is a Hilbert space with the norm
\begin{equation*}
\norm{u}_{H^{s/\gamma}_{d,\mu}}=\norm{u}_{L^2}+\sqrt{\mathscr{E}^{d,\mu}_{s/\gamma}(u)}.
\end{equation*}
The inner product is given by the usual in $L^2$ plus the bilinear form
\begin{equation*}
B^{d,\mu}_{s/\gamma}(u,v)=\iint_{X\times X}\frac{v(x)-v(y)}{d^s(x,y)}\frac{u(x)-u(y)}{d^s(x,y)}\frac{d\mu(x)d\mu(y)}{d^\gamma(x,y)}.
\end{equation*}

Aside from constant functions, there are non constant functions with finite energy and finite $L^2$ norm. Set $\Lambda^\beta(X,d)$ to denote the functions which are of H\"{o}lder-Lipschitz class. That is $u\in\Lambda^\beta(X,d)$ if $\abs{u(x)-u(y)}\leq c d^\beta(x,y)$ for $x$, $y\in X$ and some constant $C>0$.

\begin{lemma}\label{lemma:welldefinedenergy}
Let $u$ be a compactly supported $\Lambda^\beta(X,d)$ function defined on $X$. Then, with $s<\beta$ we have that $\mathscr{E}^{d,\mu}_{s/\gamma}(u)<\infty$.
\end{lemma}
\begin{proof}
Notice that since the support of $u$ is bounded, say $\supp u\subset B(x_0,R)$, the function $U(x,y)=\left(\frac{u(x)-u(y)}{d(x,y)^s}\right)^2$ is supported in $(B(x_0,R)\times X)\cup (X\times B(x_0,R))\subset$ $(B(x_0,R)\times B^c(x_0,2R))\cup (B^c(x_0,2R)\times B(x_0,R))\cup (B(x_0,2R)\times B(x_0,2R))$, we may write 
\begin{align*}
\iint_{X\times X}U(x,y)\frac{d\mu(x)d\mu(y)}{d^\gamma(x,y)}
&= \int_{d(x,x_0)<R}\int_{d(y,x_0)\geq 2R} U(x,y)\frac{d\mu(x)d\mu(y)}{d^\gamma(x,y)} \\
& \qquad + \int_{d(x,x_0)<2R}\int_{d(y,x_0)<2R}U(x,y)\frac{d\mu(x)d\mu(y)}{d^\gamma(x,y)}\\
& \qquad + \int_{d(x,x_0)\geq 2R}\int_{d(y,x_0)<R}U(x,y)\frac{d\mu(x)d\mu(y)}{d^\gamma(x,y)}\\
&= I + II + III.
\end{align*}
For $I$ we have the bound
\begin{align*}
I &\leq 4\norm{u}^2_{\infty}\int_{x\in B(x_0,R)}\left(\int_{y\notin B(x_0,2R)}\frac{d\mu(y)}{d(x,y)^{\gamma+2s}}\right) d\mu(x)\\
&\leq c\norm{u}^2_{\infty}\int_{x\in B(x_0,R)}\left(\int_{y\notin B(x_0,2R)}\frac{d\mu(y)}{d(x_0,y)^{\gamma+2s}}\right) d\mu(x)\\
&\leq c\norm{u}^2_{\infty} R^{1-2s}.
\end{align*}
In the second inequality we have used that $d(x_0,y)<2 d(x,y)$. In fact, $d(x_0,y)\leq d(x_0,x)+d(x,y)<R+d(x,y)<\frac{d(x_0,y)}{2}+d(x,y)$. The third integral $III$ can be bounded similarly. For the second we use the Lipschitz condition for $u$ which gives $U(x,y)\leq \abs{u}^2_{\Lambda^\beta}d(x,y)^{2(\beta-s)}$. Hence
\begin{align*}
II &\leq c \abs{u}^2_{\Lambda^\beta}\iint_{B(x_0,2R)\times B(x_0,2R)}d(x,y)^{2(\beta-s)-\gamma} d\mu(x)d\mu(y)\\
&= c\int_{d(x,x_0)<2R}\left(\int_{d(y,x_0)<2R}d(x,y)^{2(\beta-s)-\gamma}d\mu(y)\right) d\mu(x)\\
&\leq c\int_{d(x,x_0)<2R}\left(\int_{d(y,x)<4R}d(x,y)^{2(\beta-s)-\gamma}d\mu(y)\right) d\mu(x)\\
&\leq \widetilde{c}R R^{2(\beta-s)}\\
&= \widetilde{c}\abs{u}^2_{\Lambda^\beta}R^{2(\beta-s)+1}.
\end{align*}
\end{proof}
On the other hand, the characterization of Lipschitz integral spaces given in \cite{MaSe79Lip} in the abstract setting of spaces of homogeneous type, leads to the next result which in particular proves the continuity of the functions with finite energy for some range of parameters.
\begin{theorem}\label{thm:energyLipschitz}
Let $(X,d,\mu)$ be Ahlfors $\gamma$-regular metric space and $\tfrac{\gamma}{2}<s<1$. If $\mathscr{E}^{d,\mu}_{s/\gamma}(u)$ is finite then $u$ belongs to $\Lambda^{s-\tfrac{\gamma}{2}}(X,d)$. Moreover, the $\Lambda^{s-\tfrac{\gamma}{2}}(X,d)$ seminorm of $u$ is bounded above by a constant times $\sqrt{\mathscr{E}^{d,\mu}_{s/\gamma}(u)}$.
\end{theorem}
\begin{proof}
With $m_B(u)$ we denote the mean value of $u$ on the ball $B$. Then by Schwartz inequality and Lemma~\ref{lem:localglobalintegrability}, with $B=B(x_0,r)$, we have
\begin{align*}
\mu(B)^{-1-2(\tfrac{s}{\gamma}-\tfrac{1}{2})}\int_B &\abs{u(x)-m_B(u)}^2 d\mu(x) =\mu(B)^{-1-2(\tfrac{s}{\gamma}-\tfrac{1}{2})}\int_{x\in B}\abs{\frac{1}{\mu(B)}\int_{y\in B}(u(x)-u(y)) d\mu(y)}^2 d\mu(x)\\
&=\mu(B)^{-2(1+\tfrac{s}{\gamma})}\int_{x\in B}\abs{\int_{y\in B}\frac{u(x)-u(y)}{d^s(x,y)}\frac{1}{d^{\tfrac{\gamma}{2}}(x,y)}d^{\tfrac{\gamma}{2}+s}(x,y) d\mu(y)}^2 d\mu(x)\\
&\leq \mu(B)^{-2(1+\tfrac{s}{\gamma})}\int_{x\in B}\left(\int_{y\in B}\abs{\frac{u(x)-u(y)}{d(x,y)^s}}^2\frac{d\mu(y)}{d^\gamma(x,y)}\right)
\left(\int_{y\in B}d(x,y)^{\gamma+2s}d\mu(y)\right) d\mu(x)\\
&\leq \mu(B)^{-2(1+\tfrac{s}{\gamma})}\int_{x\in B}\left(\int_{y\in B}\abs{\frac{u(x)-u(y)}{d(x,y)^s}}^2\frac{d\mu(y)}{d^\gamma(x,y)}\right)
\left(\int_{y\in B(x,2r)}d(x,y)^{\gamma+2s}d\mu(y)\right) d\mu(x)\\
&= C r^{-2\gamma(1+\tfrac{s}{\gamma})} r^{2\gamma+2s}\mathscr{E}^{d,\mu}_{s/\gamma}(u)\\
&= C\mathscr{E}^{d,\mu}_{s/\gamma}(u).
\end{align*}
Hence 
\begin{equation*}
\left(\frac{1}{\mu(B)}\int_B \abs{u(x)-m_B(u)}^2 d\mu(x)\right)^{\tfrac{1}{2}}\leq C\sqrt{\mathscr{E}^{d,\mu}_{s/\gamma}(u)}\,\mu(B)^{\tfrac{s}{\gamma}-\tfrac{1}{2}},
\end{equation*}
or in the notation of \cite{MaSe79Lip} $u\in Lip(\tfrac{s}{\gamma}-\tfrac{1}{2},2)$. From  Theorem~4 in \cite{MaSe79Lip} we have that $\abs{u(x)-u(y)}\leq C\sqrt{\mathscr{E}^{d,\mu}_{s/\gamma}(u)}\mu(B)^{\tfrac{s}{\gamma}-\tfrac{1}{2}}$ for every ball containing $x$ and $y$. Since $B(x,2d(x,y))$ contains $x$ and $y$ and $\mu(B(x,2d(x,y)))\simeq d^\gamma(x,y)$ we get that
\begin{equation*}
\abs{u(x)-u(y)}\leq C\sqrt{\mathscr{E}^{d,\mu}_{s/\gamma}(u)}\,d(x,y)^{s-\tfrac{\gamma}{2}},
\end{equation*}
as desired.
\end{proof}
\begin{corollary}\label{coro:finiteenergycontinuousf}
If $\tfrac{\gamma}{2}<s<1$ and $\mathscr{E}^{d,\mu}_{s/\gamma}(u)<\infty$ then the function $u$ is continuous.
\end{corollary}

\section{Existence of weak solutions and Green's functions}\label{sec:weaksolutionsandGreenfunction}

In this section we prove the existence of weak solution for the problem
\begin{equation*}
\left\{
\begin{array}{ll}
 D^{2s}_d u = f\\
u\in H^{s/\gamma,0}_{d,\mu} \,  
\end{array}
\right.
\end{equation*}
for $f$ in the dual of $H^{s/\gamma,0}_{d,\mu}$. Here $D^{2s}_d u(x)=\int\frac{u(x)-u(y)}{d^{2s}(x,y)}\frac{d\mu(y)}{d^\gamma(x,y)}$ and $H^{s/\gamma,0}_{d,\mu}$ is any of the spaces of coercivity for $\mathscr{E}^{d,\mu}_{s/\gamma}$ introduced in Theorem~\ref{thm:CoercivityPi} and Corollary~\ref{coro:CoercivityPj}. The result follows from Lax-Milgram Theorem and Theorem~\ref{thm:CoercivityPi} in Section~\ref{sec:energyandspacescoercivity}.

\begin{theorem}\label{thm:LaxMilgramforB}
The bilinear form
\begin{equation*}
B^{d,\mu}_{s/\gamma}(u,v)=\int_X\int_{X}\frac{u(x)-u(y)}{d^s(x,y)}\frac{v(x)-v(y)}{d^s(x,y)}\frac{d\mu(x)d\mu(y)}{d^\gamma (x,y)}.
\end{equation*}
is bounded and coercive on any space $H^{s/\gamma,0}_{d,\mu}$. Then for each $f$ in the dual of $H^{s/\gamma,0}_{d,\mu}$ there exists a unique $u\in H^{s/\gamma,0}_{d,\mu}$ such that $B^{d,\mu}_{s/\gamma}(u,\cdot)=f$. Precisely, there exists a unique $u\in H^{s/\gamma,0}_{d,\mu}$ such that $B^{d,\mu}_{s/\gamma}(u,v)=\proin{f}{v}$ for every $v\in H^{s/\gamma,0}_{d,\mu}$.
\end{theorem}
\begin{proof}
The proof of the boundedness of $B^{d,\mu}_{s/\gamma}$ follows from Schwartz inequality in the space $L^2(X\times X,\tfrac{d\mu(x)d\mu(y)}{d^\gamma(x,y)})$, since
\begin{equation*}
\abs{B^{d,\mu}_{s/\gamma}(u,v)}\leq \sqrt{\mathscr{E}^{d,\mu}_{s/\gamma}(u)\mathscr{E}^{d,\mu}_{s/\gamma}(v)}\leq \norm{u}_{H^{s/\gamma,0}_{d,\mu}}\norm{v}_{H^{s/\gamma,0}_{d,\mu}}.
\end{equation*}
The coercivity is proved in Theorem~\ref{thm:CoercivityPi}, since $\mathscr{E}^{d,\mu}_{s/\gamma}(u)=B^{d,\mu}_{s/\gamma}(u,u)$. Hence we can apply Lax-Milgram Theorem. Given $f$ in the dual of $H^{s/\gamma,0}_{d,\mu}$, there exists a unique $u\in H^{s/\gamma,0}_{d,\mu}$ such that
\begin{equation*}
B^{d,\mu}_{s/\gamma}(u,v)=\proin{f}{v}
\end{equation*}
for every $v\in H^{s/\gamma,0}_{d,\mu}$, as desired.
\end{proof}

In particular, from Corollary~\ref{coro:CoercivityPj}, for $s>\tfrac{\gamma}{2}$ and $x\in X$, the functional $\proin{\delta_x}{\varphi}=\varphi(x)$ is well defined for $\varphi\in H^{s/\gamma,0}_{d,\mu}$. 

\begin{theorem}\label{thm:existenceGreenFunction}
For $s>\tfrac{\gamma}{2}$ there exists a function $\mathcal{G}(x,y)$ defined on $X\times X$ such that for each $x\in X$, $\mathcal{G}(x,\cdot)$ belongs to $H^{s/\gamma,0}_{d,\mu}$ as a function of $y$ and satisfies
\begin{equation*}
B^{d,\mu}_{s/\gamma}(\mathcal{G}(x,\cdot),v)=v(x),
\end{equation*}
for every $v\in H^{s/\gamma,0}_{d,\mu}$.
\end{theorem}
\begin{proof}
Take $f=\delta_x$ in Theorem~\ref{thm:LaxMilgramforB}.
\end{proof}

Since the Euler-Lagrange operator associated to the bilinear form $B^{d,\mu}_{s/\gamma}$ is $D^{2s}_d$ we may at least formally write that, in the weak sense the Green function solves the problem
\begin{equation*}
\left\{
\begin{array}{ll}
 D^{2s}_d \mathcal{G}(x,\cdot) = \delta_x\\
 P_0\mathcal{G}(x,\cdot) = 0,  
\end{array}
\right.
\end{equation*}
where $P_0$ is the projector of $L^2$ onto $H^{s/\gamma,0}_{d,\mu}$.

Let us notice that in the example of the Sierpinski quadrant given in the introduction, $\gamma=\tfrac{\log 3}{\log 2}$. So that, for $\tfrac{\gamma}{2}<s<1$, Theorem~\ref{thm:existenceGreenFunction} provides Green functions with vanishing small resolution in this setting.




\providecommand{\bysame}{\leavevmode\hbox to3em{\hrulefill}\thinspace}
\providecommand{\MR}{\relax\ifhmode\unskip\space\fi MR }
\providecommand{\MRhref}[2]{%
	\href{http://www.ams.org/mathscinet-getitem?mr=#1}{#2}
}
\providecommand{\href}[2]{#2}



\bigskip
\noindent{\footnotesize
	\textsc{Instituto de Matem\'{a}tica Aplicada del Litoral, UNL, CONICET.}

	\smallskip
	\noindent\textmd{CCT CONICET Santa Fe, Predio ``Alberto Cassano'', Colectora Ruta Nac.~168 km 0, Paraje El Pozo, S3007ABA Santa Fe, Argentina.}
}
%

\end{document}